\documentclass[final,leqno]{siamltex}
\usepackage{epsfig}
\usepackage{amsmath}
\usepackage{amssymb}
\usepackage{tikz}
\usepackage{graphicx}
\usepackage[notcite,notref]{showkeys}
\newtheorem{algorithm}{Algorithm}
\newcommand{\bq}{{\bf q}}
\newcommand{\bn}{{\bf n}}
\newcommand{\bx}{{\bf x}}
\newcommand{\bv}{{\bf v}}
\newcommand{\bw}{{\bf w}}

\def\T{{\mathcal T}}

\def\Q{{\mathbb Q}}

\def\l{{\langle}}
\def\r{{\rangle}}

\def\bn{{\bf n}}
\def\bq{{\bf q}}

  \def\b#1{\mathbf{#1}} 
\def\a#1{\begin{align*}#1\end{align*}} \def\an#1{\begin{align}#1\end{align}} 
 \def\div{\operatorname{div}}
\def\p#1{\begin{pmatrix}#1\end{pmatrix}}

\def\bbQ{\mathbb{Q}}
\newcommand{\pT}{{\partial T}}

\def\3bar{{|\hspace{-.02in}|\hspace{-.02in}|}}
\title{A stabilizer free weak Galerkin finite element
   method on polytopal mesh: Part III}

\author{Xiu Ye\thanks{Department of
Mathematics, University of Arkansas at Little Rock, Little Rock, AR
72204 (xxye@ualr.edu). This research was supported in part by
National Science Foundation Grant DMS-1620016.}
\and
Shangyou Zhang\thanks{Department of
Mathematical Sciences, University of Delaware, Newark, DE 19716 (szhang@udel.edu).}
}

\begin{document}

\maketitle

\begin{abstract}
A weak Galerkin (WG) finite element method without stabilizers was introduced in [J. Comput. Appl. Math., 371 (2020). arXiv:1906.06634] on polytopal mesh. Then it was improved in
[arXiv:2008.13631] with order one superconvergence.
The goal of this paper is to develop a new stabilizer free WG method on polytopal mesh. This method has convergence rates two orders higher than the optimal convergence rates for the corresponding WG solution in both an energy norm and the $L^2$ norm.
The numerical examples are tested for low and  high order elements in two and three dimensional  spaces.
\end{abstract}


\pagestyle{myheadings}

\section{Introduction}\label{Section:Introduction}

A weak Galerkin finite element method without stabilizers on polytopal mesh has been developed in \cite{sf-wg} and improved in \cite{sf-wg-part-II} for the Poisson equation:
\begin{eqnarray}
-\Delta u&=&f\quad \mbox{in}\;\Omega,\label{pde}\\
u&=&0\quad\mbox{on}\;\partial\Omega.\label{bc}
\end{eqnarray}
Here  $\Omega$ is a  polygonal or polyhedral domain.
%The WG method developed in \cite{sf-wg,sf-wg-part-II} has the following simple formulation without any stabilizers:
%\begin{equation}\label{dfe}
%(\nabla_w u_h,\nabla_w v)=(f,v),
%\end{equation}
%where  $\nabla_w$ is weak gradient.
A stabilizing terms  are often needed in finite element methods
with discontinuous approximations  to ensure weak continuity of discontinuous functions 
across element boundaries. Development of  stabilizer free discontinuous  finite element method is desirable because it  simplifies
finite element formulation and reduces complexity of coding. However it is a difficult task on polygonal or polyhedral mesh.

The idea of removing stabilizers for the WG methods in \cite{sf-wg,sf-wg-part-II}  is how to approximate weak gradient $\nabla_w$.  A  polynomial of degree $j$ is used in \cite{sf-wg,cdg2} to approximate  weak gradient $\nabla_w$. Here $j=k+n-1$ and $n$ is the number of sides of  polytopal element. The authors in \cite{aw,aw1} have relaxed the requirement of  polynomial degree of approximation. Rational function Wachspress coordinates \cite{cw} are used in \cite{liu, mu} to approximate weak gradient. A new stabilizer free WG method has been introduced recently in \cite{sf-wg-part-II} on polytopal mesh, which has order one superconvergence.  Piecewise low order polynomials on a polytopal element are employed  for $\nabla_w$ in \cite{sf-wg-part-II},  instead of using one piece high order polynomials in \cite{sf-wg}.

The goal of this paper is to introduce a new WG method without stabilizers on polygonal/polyhedral mesh, which has order two superconvergence, compared with order one superconvergence of the WG method  in \cite{sf-wg-part-II}.
The superconvergence is proved for the numerical approximation in both an energy norm and the $L^2$ norm.  Many numerical tests are conducted for the new WG elements of different degrees in two and three dimensional spaces.

\section{Weak Galerkin Finite Element Method}\label{Section:wg-fem}

Let ${\cal T}_h$ be a partition of the domain $\Omega$ consisting of
polygons/polyhedra  that satisfies
a set of conditions defined in \cite{wymix}. Let ${\cal E}_h$ denote
the set of all edges or flat faces in ${\cal T}_h$, and ${\cal
E}_h^0={\cal E}_h\backslash\partial\Omega$ denote the set of all
interior edges or flat faces. Denote by $h_T$ the diameter of $T\in\T_h$ and mesh size $h=\max_{T\in\T_h} h_T$
for ${\cal T}_h$. Let $P_k(T)$ consist all the polynomials on $T$ with degree less or equal to $k$.

We define $V_h$  the weak Galerkin finite element space for a given integer $k \ge 0$  as follows
\begin{equation}\label{vhspace}
V_h=\{v=\{v_0,v_b\}:\; v_0|_T\in P_k(T),\ v_b|_e\in P_{k+1}(e),\ e\subset\pT,\;  T\in \T_h\}.
\end{equation}
Define  $V_h^0$ a subspace of $V_h$  as
\begin{equation}\label{vh0space}
V^0_h=\{v: \ v\in V_h,\  v_b=0 \mbox{ on } \partial\Omega\}.
\end{equation}
%We would like to emphasize that any function $v\in V_h$ has a single
%value $v_b$ on each edge $e\in\E_h$.

A weak gradient $\nabla_wv$ for $v=\{v_0,v_b\}\in V_h$ is defined as a piecewise polynomial such that  $\nabla_w v|_T \in \Lambda_k(T)$ for $T\in\T_h$ and
\begin{equation}\label{d-d}
  (\nabla_w v, \bq)_T = -(v_0, \nabla\cdot \bq)_T+ \langle v_b, \bq\cdot\bn\rangle_{\partial T}\qquad
   \forall \bq\in \Lambda_k(T).
\end{equation}
We will define $\Lambda_k(T)$ in the next section.

\begin{algorithm}
A numerical approximation for (\ref{pde})-(\ref{bc}) is seeking $u_h=\{u_0,u_b\}\in V_h^0$ that
satisfies  the following equation:
\begin{equation}\label{wg}
(\nabla_wu_h,\nabla_wv)=(f,\; v_0) \quad\forall v=\{v_0,v_b\}\in V_h^0.
\end{equation}
\end{algorithm}

The following notations will be adopted,
\begin{eqnarray*}
(v,w)_{\T_h} &=&\sum_{T\in\T_h}(v,w)_T=\sum_{T\in\T_h}\int_T vw d\bx,\\
 \l v,w\r_{\partial\T_h}&=&\sum_{T\in\T_h} \l v,w\r_\pT=\sum_{T\in\T_h} \int_\pT vw ds.
\end{eqnarray*}

\section{Existence and Uniqueness}
In this section, we will investigate the well posedness of the WG method.
The space $H(div;\Omega)$ is defined as 
\[
H(div; \Omega)=\left\{ \bv\in [L^2(\Omega)]^d:\; \nabla\cdot\bv \in L^2(\Omega)\right\}.
\]
For any $T\in\T_h$, it can be divided in to a set of disjoint triangles/tetrahedrons $T_i$ with $T=\cup T_i$.  Then we define $\Lambda_h(T)$ for the approximation of weak gradient on element $T$  as  
\begin{eqnarray}
\Lambda_{k}(T)=\{\bv\in H(div,T):&&\ \bv|_{T_i}\in [P_{k+1}(T_i)]^d,\;\;\nabla\cdot\bv\in P_k(T).\label{lambda}\\
&&\bv\cdot\bn|_e\in P_{k+1}(e),\;e\subset\pT\},\nonumber
\end{eqnarray}

\begin{theorem}\label{zhang}
For $\tau\in H(div,\Omega)$, there exists a projection $\Pi_h$ with $\Pi_h\tau\in H(div,\Omega)$ satisfying $\Pi_h\tau|_T\in \Lambda_k(T)$ and
\begin{eqnarray}
(\nabla\cdot\tau,\;v_0)_T&=&(\nabla\cdot\Pi_h\tau,\;v_0)_T, \label{key2}\\
-(\nabla\cdot\tau, \;v_0)_{\T_h}&=&(\Pi_h\tau, \;\nabla_w v)_{\T_h},\label{key1}\\
\|\Pi_h\tau-\tau\|&\le& Ch^{k+2}|\tau|_{k+2}.\label{key3}
\end{eqnarray}
\end{theorem}
We will prove this important theorem in Section \ref{proof}.

\smallskip

Let $Q_0$ be the element-wise defined $L^2$ projection onto $P_k(T)$ on each $T\in\T_h$. Similarly let $Q_b$ be the $L^2$ projection onto $P_{k+1}(e)$ with $e\subset\partial T$. Let $\Q_h$ be the element-wise defined $L^2$ projection onto $\Lambda_k(T)$ on each element $T$. Finally we define $Q_hu=\{Q_0u,Q_bu\}\in V_h$.

\smallskip

\begin{lemma}
Let $\phi\in H^1(\Omega)$ and $v\in V_h^0$, then we have,
\begin{eqnarray}
\nabla_w Q_h\phi& =&\Q_h\nabla\phi.\label{key0}
%\nabla_w\phi &=&\Q_h\nabla\phi.\label{key}
\end{eqnarray}
\end{lemma}
\begin{proof}
It follows from (\ref{d-d}), the definition of $\Lambda_k(T)$ and  integration by parts that for
any $\bq\in \Lambda_k(T)$
\begin{eqnarray*}
(\nabla_w Q_h\phi,\bq)_T &=& -(Q_0\phi,\nabla\cdot\bq)_T
+\langle Q_b\phi,\bq\cdot\bn\rangle_{\pT}\\
&=& -(\phi,\nabla\cdot\bq)_T
+\langle \phi,\bq\cdot\bn\rangle_{\pT}\\
&=&(\nabla \phi,\bq)_T\\
&=&(\Q_h\nabla\phi,\bq)_T.
\end{eqnarray*}
Thus we have proved the identity (\ref{key0}).
\end{proof}

We define a semi-norm $\3bar\cdot\3bar$ for any $v\in V_h$ as
\begin{equation}\label{3barnorm}
\3bar v\3bar^2=(\nabla_wv,\nabla_wv)_{\T_h}.
\end{equation}
We define another semi-norm as
\begin{equation}\label{norm}
\|v\|_{1,h} = \left( \sum_{T\in\T_h}\left(\|\nabla
v_0\|_T^2+h_T^{-1} \|  v_0-v_b\|^2_\pT\right) \right)^{\frac12}.
\end{equation}
Obviously,  $\|v\|_{1,h}$ define a norm for  $v\in V_h^0$. We prove the equivalence of $\3bar\cdot\3bar$ and $\|\cdot\|_{1,h}$ in the following lemma.

\smallskip

\begin{lemma}\label{lemma11} 
There exist two positive constants $C_1$ and $C_2$ so
that the following inequalities hold true for any $v=\{v_0,v_b\}\in V_h$, 
\begin{equation}\label{happy}
C_1 \|v\|_{1,h}\le \3bar v\3bar \leq C_2 \|v\|_{1,h}.
\end{equation}
\end{lemma}

\begin{lemma}\label{lemma22}
The weak Galerkin finite element scheme (\ref{wg}) has a unique
solution.
\end{lemma}

\medskip

The proofs for Lemma \ref{lemma11} and Lemma \ref{lemma22} are similar to the ones in \cite{sf-wg-part-II}.

\section{Error Estimates}

In this section, we will derive superconvergence for the WG finite element approximation $u_h$ in both an energy norm and the $L^2$ norm.

\subsection{Error Estimates in Energy Norm}
We start this subsection by deriving an error equation that $\epsilon_h=Q_hu-u_h$ satisfies. First we define
\begin{eqnarray}
\ell(u,v)&=&(\bbQ_h\nabla u-\Pi_h\nabla u, \nabla_w v)_{\T_h}.\label{nnn}
\end{eqnarray}

\begin{lemma}
Let $\ell(u,v)$ defined in (\ref{nnn}). Then we have
\begin{eqnarray}
(\nabla_w \epsilon_h,\nabla_wv)=\ell(u,v)\quad\forall v\in V_h^0.\label{ee}
\end{eqnarray}
\end{lemma}

\begin{proof}
For $v=\{v_0,v_b\}\in V_h^0$, testing (\ref{pde}) by  $v_0$  and using (\ref{key1}),  we have
\begin{equation}\label{m1}
(f, v_0)=-(\nabla\cdot\nabla u, v_0)=(\Pi_h\nabla u, \nabla_w v)_{\T_h}.
\end{equation}
It follows from (\ref{key0}) and (\ref{m1})
\begin{equation}\label{j2}
(\nabla_wQ_h u, \nabla_w v)=(f, v_0)+\ell(u,v).
\end{equation}
Subtracting (\ref{wg}) from (\ref{j2}) gives the error equation,
\begin{eqnarray*}
(\nabla_w\epsilon_h,\nabla_wv)=\ell(u,v)\quad \forall v\in V_h^0.
\end{eqnarray*}
This completes the proof of the lemma.
\end{proof}

For any function $\varphi\in H^1(T)$, the following trace
inequality holds true (see \cite{wymix} for details):
\begin{equation}\label{trace}
\|\varphi\|_{e}^2 \leq C \left( h_T^{-1} \|\varphi\|_T^2 + h_T
\|\nabla \varphi\|_{T}^2\right).
\end{equation}

\begin{theorem} Let $u_h\in V_h$ and $u\in H^{k+3}(\Omega)$ be the solutions of (\ref{wg}) and (\ref{pde}), respectively.  Then we have
\begin{equation}\label{err1}
\3bar Q_hu-u_h\3bar \le Ch^{k+2}|u|_{k+3}.
\end{equation}
\end{theorem}
\begin{proof}
Letting $v=\epsilon_h$ in (\ref{ee}), we arrive at
\begin{eqnarray}
\3bar \epsilon_h\3bar^2&=&\ell(u,\epsilon_h).\label{eee1}
\end{eqnarray}
The definitions of $\bbQ_h$ and $\Pi_h$ yield
\begin{eqnarray}
|\ell(u,\epsilon_h)|&=&|(\bbQ_h\nabla u-\Pi_h\nabla u, \nabla_w \epsilon_h)_{\T_h}|\nonumber\\
&\le&(\sum_T\|\bbQ_h\nabla u-\Pi_h\nabla u\|_T)^{1/2}\3bar\epsilon_h\3bar\nonumber\\
&\le&(\sum_T\|\bbQ_h\nabla u-\nabla u+\nabla u-\Pi_h\nabla u\|_T)^{1/2}\3bar\epsilon_h\3bar\nonumber\\
&\le& Ch^{k+2}|u|_{k+3}\3bar \epsilon_h\3bar.\label{eee3}
\end{eqnarray}
It follows (\ref{eee1}) and (\ref{eee3}), 
\[
\3bar \epsilon_h\3bar \le Ch^{k+2}|u|_{k+3}.
\]
The proof of the theorem is completed.
\end{proof}

\subsection{Error Estimates in $L^2$ Norm}
First notice $\epsilon_h=\{\epsilon_0,\epsilon_b\}=Q_hu-u_h$.
We use standard duality argument to derive $L^2$ error estimate.
The dual problem is finding $\phi\in H_0^1(\Omega)$ such that 
\begin{eqnarray}
-\Delta\phi&=& \epsilon_0\quad
\mbox{in}\;\Omega,\label{dual}
%w&=&0\quad \mbox{on}\; \partial\Omega,\label{dual-BC}
\end{eqnarray}
and the following $H^{2}$ regularity holds
\begin{equation}\label{reg}
\|\phi\|_2\le C\|\epsilon_0\|.
\end{equation}

\begin{lemma}
The following equation holds true for any $v\in V_h^0$,
\begin{eqnarray}
(\nabla_w Q_h\phi,\nabla_wv)_{\T_h}=(\epsilon_0,v_0)+\ell_1(\phi,v),\label{ee1}
\end{eqnarray}
where
\begin{eqnarray*}
\ell_1(\phi,v)&=& \langle (\nabla \phi-\Q_h\nabla \phi)\cdot\bn,v_0-v_b\rangle_{\partial\T_h}.
%\ell_2(u,v)&=& (a\nabla_wQ_h u-a\Q_h\nabla u,\nabla_wv)_{\T_h}.
\end{eqnarray*}
\end{lemma}

\begin{proof}
For $v=\{v_0,v_b\}\in V_h^0$, testing (\ref{dual}) by  $v_0$ gives
\begin{equation}\label{d1}
-(\Delta \phi, v_0)=(\epsilon_0,v_0).
\end{equation}
Using integration by parts and the fact that
$\sum_{T\in\T_h}\langle \nabla \phi\cdot\bn, v_b\rangle_\pT=0$,  we arrive at
\begin{equation}\label{mm1}
-(\Delta \phi, v_0)=(\nabla \phi,\nabla v_0)_{\T_h}- \langle
\nabla \phi\cdot\bn,v_0-v_b\rangle_{\partial\T_h}.
\end{equation}
It follows from integration by parts, (\ref{d-d}) and (\ref{key0})  that
\begin{eqnarray}
(\nabla \phi,\nabla v_0)_{\T_h}&=&(\Q_h\nabla  \phi,\nabla v_0)_{\T_h}\nonumber\\
&=&-(v_0,\nabla\cdot (\Q_h\nabla \phi))_{\T_h}+\langle v_0, \Q_h\nabla \phi\cdot\bn\rangle_{\partial\T_h}\nonumber\\
&=&(\Q_h\nabla \phi, \nabla_w v)_{\T_h}+\langle v_0-v_b,\Q_h\nabla \phi\cdot\bn\rangle_{\partial\T_h}\nonumber\\
&=&( \nabla_w Q_h\phi, \nabla_w v)_{\T_h}+\langle v_0-v_b,\Q_h\nabla \phi\cdot\bn\rangle_{\partial\T_h}.\label{j1}
\end{eqnarray}
Combining (\ref{mm1}) and (\ref{j1}) gives
\begin{eqnarray}
-(\Delta \phi, v_0)=(\nabla_w Q_h\phi,\nabla_w v)_{\T_h}-\ell_1(\phi,v).\label{d2}
\end{eqnarray}
Combining (\ref{d2}) and (\ref{d1}) yields
\begin{eqnarray}
(\nabla_w Q_h\phi,\nabla_w v)_{\T_h}=(\epsilon_0,v_0)+\ell_1(\phi,v).\label{d3}
\end{eqnarray}
This completes the proof of the lemma.
\end{proof}

By the same argument as (\ref{d2}), (\ref{ee}) has another form as
\begin{eqnarray}
(\nabla_w \epsilon_h,\; \nabla_w v)_{\T_h}&=&\ell_1(u,v).\label{eee}
\end{eqnarray}

\begin{theorem} Let $u_h\in V_h$ be the weak Galerkin finite element solution of (\ref{wg}). Assume that the
exact solution $u\in H^{k+3}(\Omega)$ and (\ref{reg}) holds true.
 Then, there exists a constant $C$ such that for $k\ge 1$
\begin{equation}\label{err2}
\|Q_0u-u_0\| \le Ch^{k+3}|u|_{k+3}.
\end{equation}
\end{theorem}

\begin{proof}
Letting $v=\epsilon_h$ in (\ref{ee1}) gives
\begin{eqnarray}
\|\epsilon_0\|^2=(\nabla_w Q_h\phi,\nabla_w \epsilon_h)_{\T_h}-\ell_1(\phi,\epsilon_h).\label{d5}
\end{eqnarray}
Letting $v=Q_h\phi$ in (\ref{eee}) gives
\begin{eqnarray}
(\nabla_w \epsilon_h,\; \nabla_w Q_h\phi)_{\T_h}&=&\ell_1(u,Q_h\phi).\label{d6}
\end{eqnarray}
It follows from (\ref{d5}) and (\ref{d6})
\begin{eqnarray}
\|\epsilon_0\|^2=\ell_1(u,Q_h\phi)-\ell_1(\phi,\epsilon_h).\label{d7}
\end{eqnarray}
By the Cauchy-Schwarz inequality, the trace inequality (\ref{trace}) and the definitions of $\Q_h$ and $Q_h$, then
\begin{eqnarray*}
|\ell_1(u,Q_h\phi)|&\le&\left| \langle (\nabla u-\Q_h\nabla
u)\cdot\bn,\;
Q_0\phi-Q_b\phi\rangle_{\pT_h} \right|\\
%&\le& \left(\sum_{T\in\T_h}\|\nabla u-\Q_h\nabla
%u\|^2_\pT\right)^{1/2}
%\left(\sum_{T\in\T_h}\|Q_0\phi-Q_b\phi\|^2_\pT\right)^{1/2}\nonumber \\
&\le& C\left(\sum_{T\in\T_h}h_T\|\nabla u-\Q_h\nabla
u\|^2_\pT\right)^{1/2}
\left(\sum_{T\in\T_h}h_T^{-1}\|Q_0\phi-\phi\|^2_\pT\right)^{1/2} \nonumber\\
&\le&  Ch^{k+3}|u|_{k+3}|\phi|_2.\nonumber
\end{eqnarray*}
It follows from (\ref{trace}), (\ref{happy}) and (\ref{err1})
\begin{eqnarray*}
|\ell_1(\phi,\epsilon_h)|&=&\left|\sum_{T\in\T_h}\langle (\nabla \phi-\Q_h\nabla
\phi)\cdot\bn, \epsilon_0-\epsilon_b\rangle_\pT\right|\\
%&\le & C \sum_{T\in\T_h}\|\nabla \phi-\Q_h\nabla \phi\|_{\pT}
%\|\epsilon_0-\epsilon_b\|_\pT\nonumber\\
&\le & C \left(\sum_{T\in\T_h}h_T\|\nabla \phi-\Q_h\nabla \phi\|_{\pT}^2\right)^{\frac12}
\left(\sum_{T\in\T_h}h_T^{-1}\|\epsilon_0-\epsilon_b\|_\pT^2\right)^{\frac12}\\
&\le & Ch^{k+3}|u|_{k+3}|\phi|_{2}.
\end{eqnarray*}
Using the two estimates above, (\ref{d7}) becomes
$$
\|\epsilon_0\|^2 \leq C h^{k+3}|u|_{k+3} \|\phi\|_2.
$$
Combining the above inequality with
the regularity assumption (\ref{reg}), we obtain
 $$
\|\epsilon_0\|\leq C h^{k+3}|u|_{k+3},
$$
which completes the proof.
\end{proof}

\section{Proof of Theorem \ref{zhang}}\label{proof}

Theorem \ref{zhang} is a corollary of the following lemma.

\begin{lemma}
Let $\Pi_h : H(\div,\Omega) \to H(\div,\Omega) \cap \otimes \Lambda_{k-1}(T)$ be defined in
 \eqref{p-j} below.  For $\bv \in H(\div,\Omega)$ and for all $T\in\mathcal T_h$,
   we have,
\begin{eqnarray}
(\Pi_h\bv,\;\bw)_T&=&(\bv,\;\bw)_T  \quad \forall \bw\in [P_{k-2}(T)]^d,\label{P1}\\
\l \Pi_h\bv\cdot\bn,\; q\r_e &=&\l \bv\cdot\bn,\; q\r_e\quad\forall q\in P_{k}(e),
     e\subset\pT,\label{P2} \\
(\nabla\cdot\bv,\;q)_T&=&(\nabla\cdot\Pi_h\bv,\;q)_T \quad\forall q\in P_{k-1}(T), \label{P3}\\
-(\nabla\cdot\bv, \;v_0)_{\T_h}&=&(\Pi_h\bv, \;\nabla_w v)_{\T_h}
    \quad \forall v=\{v_0,v_b\}\in V_h^0, \label{P4}\\
%-(\nabla\cdot\tau, \;v_0)_{\T_h}&=&(\Pi_h\tau, \;\nabla_w v)_{\T_h}\label{key1}\\
\|\Pi_h\bv-\bv\|&\le& Ch^{k+1}|\bv|_{k+1}.\label{P5}
\end{eqnarray}
\end{lemma}

\begin{proof}  We prove the lemma in 3D. The proof for 2D lemma is similar and much
  simpler.

  We assume no additional inner edges is introduced when subdividing
   a polyhedron $T$ in to $n$ tetrahedrons $\{T_i\}$.
   That is,  we have precisely $n-1$ internal  triangles which separate $T$
    into $n$ parts.  For simple notation,  only one outside polygonal
     face $e_1$ of $T$ is subdivided in to
   $m$ triangles, $e_{1,1}, \dots, e_{1,m}$.
   For hexahedral finite elements \cite{Zhang-h,Zhang-t,Zhang-Zhang},
    a face quadrilateral can be curved, i.e., the image of a square under
    a tri-linear mapping. Such a curved face polygon is taken as two face triangles of
    a polyhedron $T$.

  On $n$ tetrahedrons,  a function of $\Lambda_k$ can be expressed as
 \an{\label{v-e}  \bv_h|_{T_{i_0}} = \sum_{i+j+l\le k}  \p{a_{1,ijl}\\a_{2,ijl}\\a_{3,ijl}} x^i y^j z^l, \
    i_0=1,...n.  }
$\bv_h|_{T}$ is determined by $n \dim [P_k]^3 =  {n(k+1)(k+2)(k+3)}/2 $
  coefficients.
For any $\bv\in H(\div;T)$, $\Pi_h \b v\in \Lambda_k(T)$ is defined by
\begin{subequations} \label{p-j}
\an{ \label{p-j-a}   \int_{e_{ij}\subset \partial T} (\Pi_h \b v-\bv) \cdot \b n_{ij}
         p_k dS & = 0 \quad \forall p_k \in P_k(e_{ij}), e_{ij}\ne e_{1,\ell},\ell\ge 2, \\
   \label{p-j-b}   \int_{e_{11}\subset \partial T} (\Pi_h \b v|_{e_{11}}-\Pi_h \b v|_{e_{1j}})
       \cdot \b n_{11}
         p_k dS & = 0 \quad \forall p_k \in P_k(e_{1j}),j=2,...,m, \\
   \label{p-j-c}   \int_{e_{ij}\subset T^0} (\Pi_h \b v-\bv) \cdot \b n_{ij}p_k  dS &=0
               \quad \forall p_k \in P_k(e_{ij}) \setminus P_0(e_{ij}), \\
   \label{p-j-d}   \int_{e_{ij}\subset T^0} [\Pi_h \b v]\cdot \b n_{ij}p_k  dS &=0
               \quad \forall p_k \in P_k(e_{ij}), \\
  \label{p-j-e}  \int_T (\Pi_h \b v -\bv )\cdot \nabla p_{k-1} d \b x &=0 \quad
                \forall p_{k-1} \in P_{k-1}(T)\setminus P_0(T), \\
    \label{p-j-f}  \int_{T_i}  (\Pi_h \b v -\bv)\cdot p d \b x &=0 \quad
                \forall p \in CP(T_i), \ i=1,...n, \\
      \label{p-j-g}  \int_{T_1} \nabla \cdot ( \Pi_h \b v|_{T_i} -\Pi_h \b v|_{T_1} ) p_{k-1} d \b x &=0 \quad
                \forall p_{k-1} \in P_{k-1}(T_1), \ i=2,...,n,
 } \end{subequations}
where $e_{ij}$ is the $j$-th face triangle of $T_i$ with a fixed normal vector $\bn_{ij}$,
        $[\cdot]$ denotes the jump on a face triangle,
    $\Pi_h\b v|_{T_i}$ is understood as a polynomial vector which can be used on
    another tetrahedron $T_1$,
     $e_{1j}\subset e_1\subset \partial T$ is a face triangle of
      $T_{i_j}$,
      $\Pi_h \b v|_{e_{1j}}$ is extended to the whole $e_1$ as
      one polynomial,
    and curl-polynomial space
\a{ CP(T_i)=\{ \b v\in & [P_k(T_i)]^3 \mid   \ \b v\cdot \b n_{ij}=0 \text{ \ on } e_{ij} \subset
    \partial T_i, \\
       & \int_{T_i} \b v \cdot \nabla p_{k-1} d \b x =0 \quad \forall p_{k-1}\in P_{k-1}(T_i)
    \},  }
where $e_{ij}$ also denotes the four face triangles of $T_i$.
The linear system \eqref{p-j} of equations has the following number of equations,
\a{  &\quad \ ( 2n+3-m)\frac{(k+1)(k+2)}2 + ( m-1 )\frac{(k+1)(k+2)}2 \\
     &\quad   \ + ( n-1 )\frac{(k+1)(k+2)-2}2 + ( n-1 )\frac{(k+1)(k+2)}2 \\
     &\quad  \ +  \frac{k(k+1)(k+2)-6}6 \\
     &\quad \ + n \Big( \frac{(k-1)k(k+1)}2 - \frac{ (k-2)(k-1)k}6 \Big) \\
     &\quad \ +(n-1)  \frac{k(k+1)(k+2)}6 \\
     &= \frac{n(k+1)(k+2)(k+3)}2,
  } which is exactly the number of coefficients for a $\bv_h$ function in \eqref{v-e}.
Thus we have a square linear system.
The square system has a unique solution if and only if the homogeneous system has the trivial
  solution.

Let $\bv=0$ in \eqref{p-j}.
By \eqref{p-j-a} and \eqref{p-j-b}, $\Pi_h\bv\cdot \b n= 0$ on the
  whole boundary $\partial T$.
By \eqref{p-j-c} and \eqref{p-j-d},  $\int_{e_{ij}}[\Pi_h \bv\cdot \b n_{ij} ] dS=0$ and
   $\int_{e_{ij}} \Pi_h \bv\cdot \b n_{ij} p dS=0$ for all $p\in P_k\setminus P_0$ on inter-element
   triangles $e_{ij}$.
By \eqref{p-j-g}, $\nabla \cdot \Pi_h\bv$ is a one-piece polynomial on the whole
   $T$.
Therefore, by \eqref{p-j-e},
    we have
\an{\label{d0} &\quad \ \int_T (\nabla \cdot \Pi_h\bv)^2 d \b x \\
 \nonumber &=\sum_{i=1}^n
      \Big(\int_{T_i} -\Pi_h\bv\cdot \nabla(\nabla \cdot  \Pi_h\bv ) d\b x
          +\int_{\partial T_i} \Pi_h\bv\cdot \b n (\nabla \cdot \Pi_h\bv) dS \Big) \\
  \nonumber  &=0.  }
That is,
\a{\nabla \cdot  \Pi_h\bv=0 \quad\text{ on } \ T.  }
Thus \an{ \label{div0} \Pi_h\bv|_{T_i} \in CP(T_i), \quad i=1,...,n. }
By \eqref{p-j-f}, $\Pi_h \bv=\b 0$. Hence $\Pi_h \bv$ is well defined.

For any $\b w \in [P_{k-2}(T)]^3$, we have $\b w = \nabla p_{k-1} + \nabla \times \b q_{k-1}$ on
  all $T_i$, where $\b q_{k-1}|_{T_i} \in [P_{k-1}(T_i)]^3$ can be chosen such that
    $\nabla \times \b q_{k-1} \in CP(T_i)$.
 By \eqref{p-j-e} and \eqref{p-j-f}, \eqref{P1} holds.

\eqref{P2} follows \eqref{p-j-a} and \eqref{p-j-b}.

Replacing one $\nabla \cdot \Pi_h\bv$ by $q$ in \eqref{d0}, \eqref{P3} follows.

It follows from (\ref{P3}) and (\ref{d-d}) that for $v=\{v_0,v_b\}\in V_h^0$
\begin{eqnarray*}
-(\nabla\cdot\bv , \;v_0)_{\T_h}&=&-(\nabla\cdot\Pi_h\bv , \;v_0)_{\T_h}\\
&=&-(\nabla\cdot\Pi_h\bv , \;v_0)_{\T_h}+\l v_b, \Pi_h\bv \cdot\bn\r_{\partial \T_h}\\
&=&(\Pi_h\bv , \;\nabla_w v)_{\T_h},
\end{eqnarray*} which proves \eqref{P4}.

On a size-$1$ $T$, by the finite dimensional norm-equivalence  and the
   shape-regularity assumption on sub-triangles,  the interpolation is stable in $L^2(T)$,
  i.e.,
\an{  \|\Pi_h \bv  \|_T \le C \|\bv \|_T.   \label{T-stable} }
After a scaling, the constant $C$ in \eqref{T-stable} remains same for all $h>0$.
Since $[P_k(T)]^3\subset \Lambda_k$ and $\Pi_h$ is uni-solvent, $\Pi_h \b v = \b v$ for all
  $\bv \in [P_k(T)]^3$.
It follows that, by $\Pi_h$'s $P_k$-polynomial preservation,
\a{ \| \Pi_h\bv  -\bv  \|^2 &\le C
    \sum_{T\in\mathcal T_h} (\| \Pi_h(\bv  -p_{k,T}) \|_T^2 + \| p_{k,T} -\bv  \|_T^2 ) \\
      &\le C
    \sum_{T\in\mathcal T_h} (C \| \bv  -p_{k,T} \|_T^2 + \| p_{k,T} -\bv  \|_T^2 ) \\
      &\le C
    \sum_{T\in\mathcal T_h} h^{2k+2}  | \bv   |_{k+1,T}^2  \\
      &=C h^{2k+2}  | \bv   |_{k+1}^2,
} where $p_{k,T}$ is the $k$-th Taylor polynomial of $\bv $ on $T$.
\end{proof}

\section{Numerical Experiments}\label{Section:numerical-experiments}

We solve the Poisson problem \eqref{pde}-\eqref{bc} on the unit square domain with the exact
  solution
\an{ \label{s-1} u=\sin(\pi x)\sin(\pi y).
  }
We compute the solution \eqref{s-1} on a type of quadrilateral grids, shown
   in Figure \ref{g-4}.
Here to avoid convergence to parallelograms under the nest refinement of quadrilaterals,
  we fix the shape of quadrilaterals on all levels of grids.
We list the computation in Table \ref{t1}.
We have two orders of superconvergence in $L^2$-norm and in $H^1$-like norm
  for all order finite elements,
  except for $P_0$ element which has only one order superconvergence in $L^2$ norm.
Both cases confirm the theoretic convergence rates.

\begin{figure}[htb]\begin{center}
\includegraphics[width=1.4in]{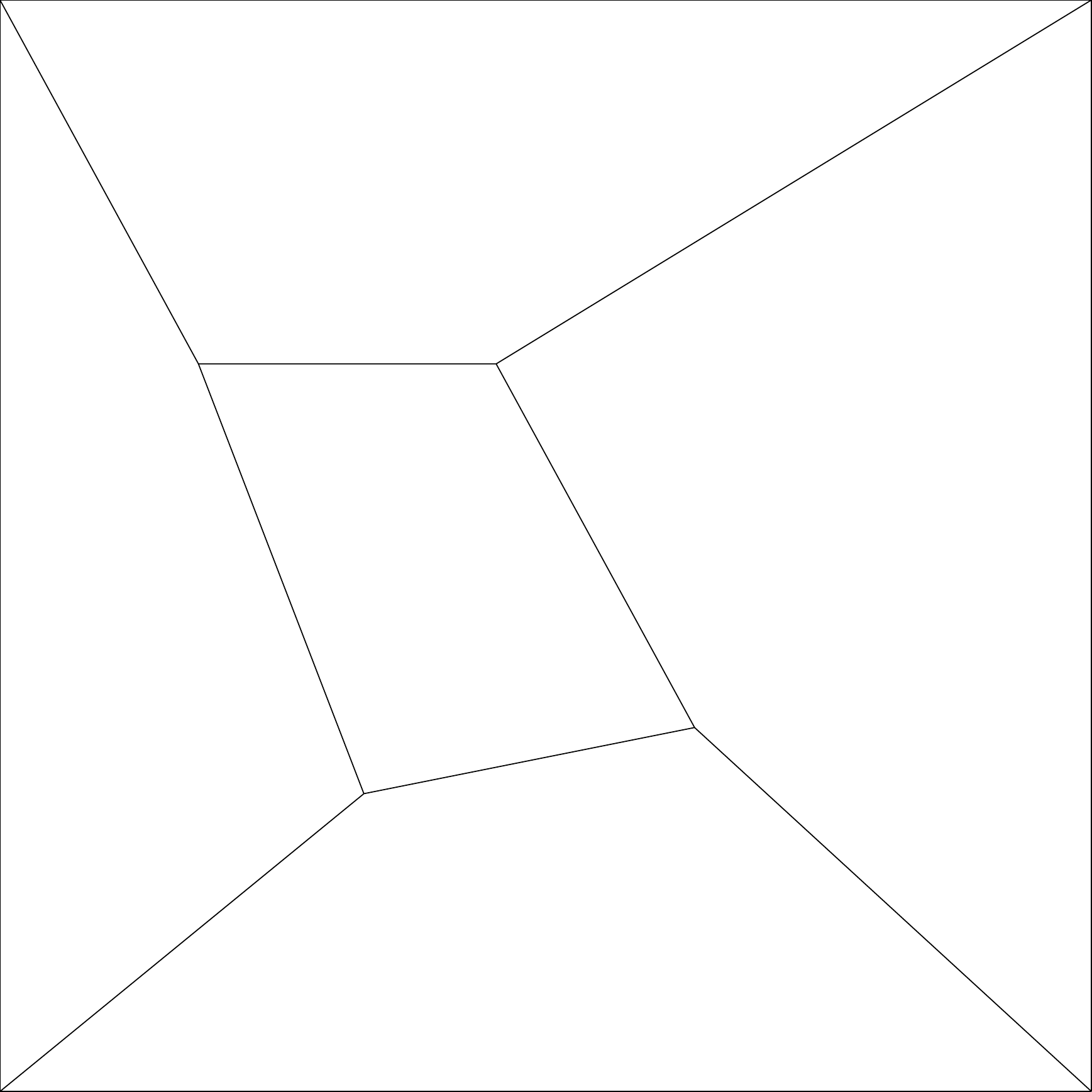} \
\includegraphics[width=1.4in]{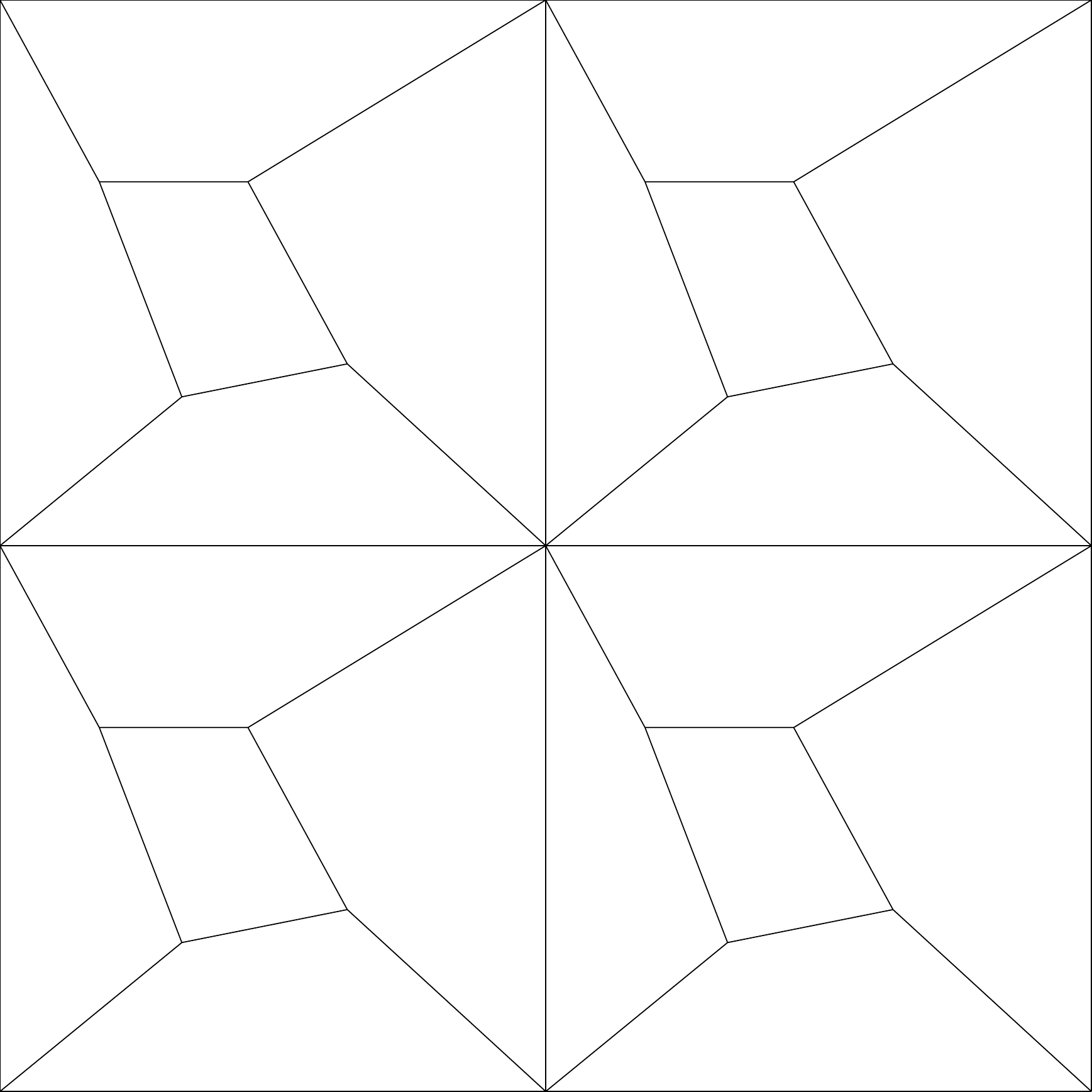} \
\includegraphics[width=1.4in]{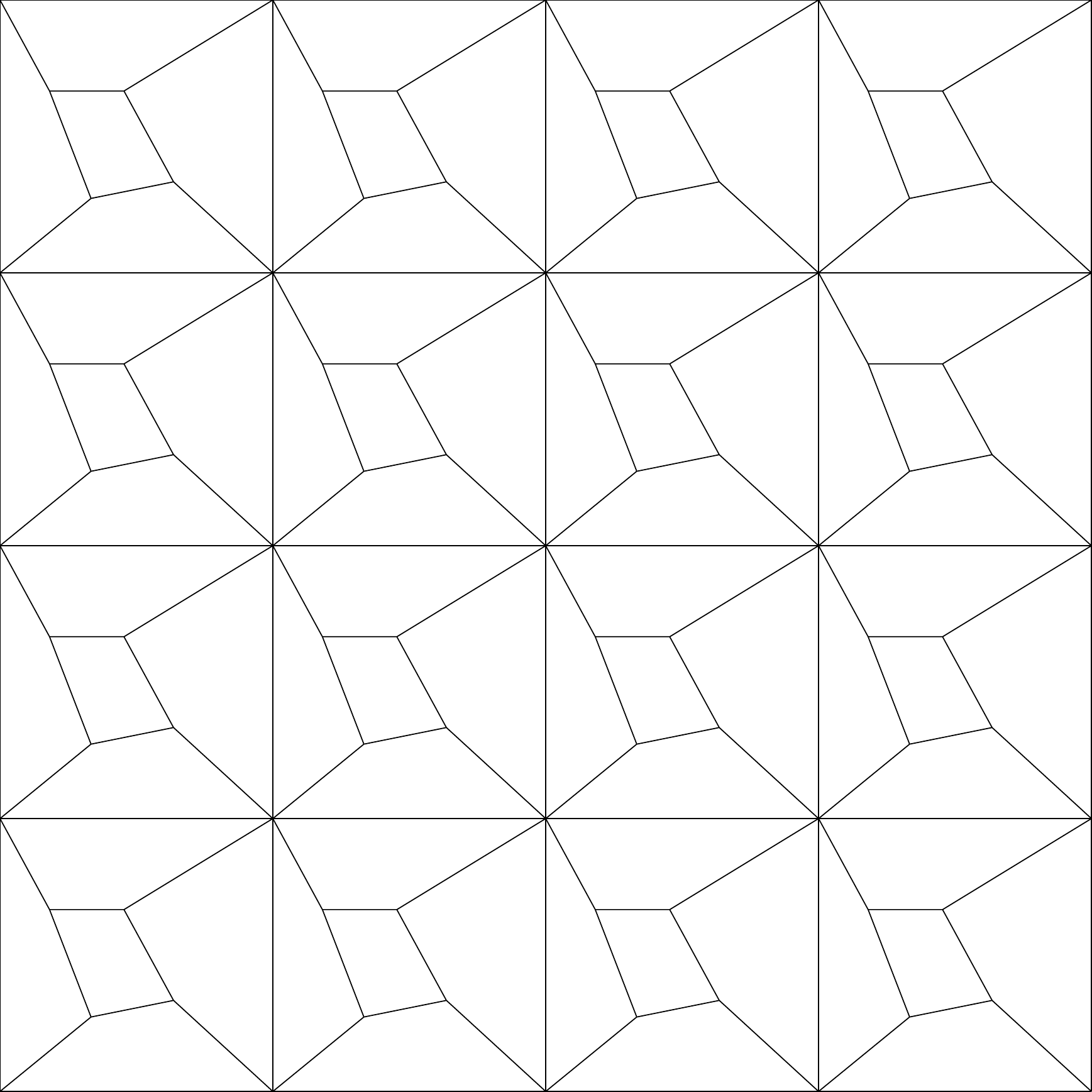}

\caption{The first three levels of quadrilateral grids, for Table \ref{t1}.  }
\label{g-4}
\end{center}
\end{figure}

\begin{table}[h!]
  \centering   \renewcommand{\arraystretch}{1.05}
  \caption{ Error profiles and convergence rates on quadrilateral
      grids shown in Figure \ref{g-4} for \eqref{s-1}. }
\label{t1}
\begin{tabular}{c|cc|cc}
\hline
level & $\|Q_h u-  u_h \|_0 $  &rate &  $\3bar Q_h u- u_h \3bar $ &rate   \\
\hline
 &\multicolumn{4}{c}{by the $P_0$-$P_1$($\Lambda_0$) WG element} \\ \hline
 6&   0.2533E-03 &  2.00&   0.1448E-02 &  2.00 \\
 7&   0.6337E-04 &  2.00&   0.3620E-03 &  2.00 \\
 8&   0.1584E-04 &  2.00&   0.9050E-04 &  2.00 \\
\hline
 &\multicolumn{4}{c}{by the $P_1$-$P_2$($\Lambda_1$) WG element} \\ \hline
 5&   0.9360E-06 &  4.00&   0.2156E-03 &  3.00 \\
 6&   0.5851E-07 &  4.00&   0.2696E-04 &  3.00 \\
 7&   0.3663E-08 &  4.00&   0.3371E-05 &  3.00 \\
 \hline
 &\multicolumn{4}{c}{by the $P_2$-$P_3$($\Lambda_2$) WG element} \\ \hline
 4&   0.7659E-06 &  4.98&   0.1487E-03 &  3.99 \\
 5&   0.2404E-07 &  4.99&   0.9307E-05 &  4.00 \\
 6&   0.7521E-09 &  5.00&   0.5819E-06 &  4.00 \\
\hline
 &\multicolumn{4}{c}{by the $P_3$-$P_4$($\Lambda_3$) WG element} \\ \hline
 2&   0.1439E-03 &  3.84&   0.9642E-02 &  3.24 \\
 3&   0.2319E-05 &  5.95&   0.3065E-03 &  4.98 \\
 4&   0.3646E-07 &  5.99&   0.9620E-05 &  4.99 \\
 \hline
\end{tabular}%
\end{table}%

\begin{figure}[htb]\begin{center}
\includegraphics[width=1.4in]{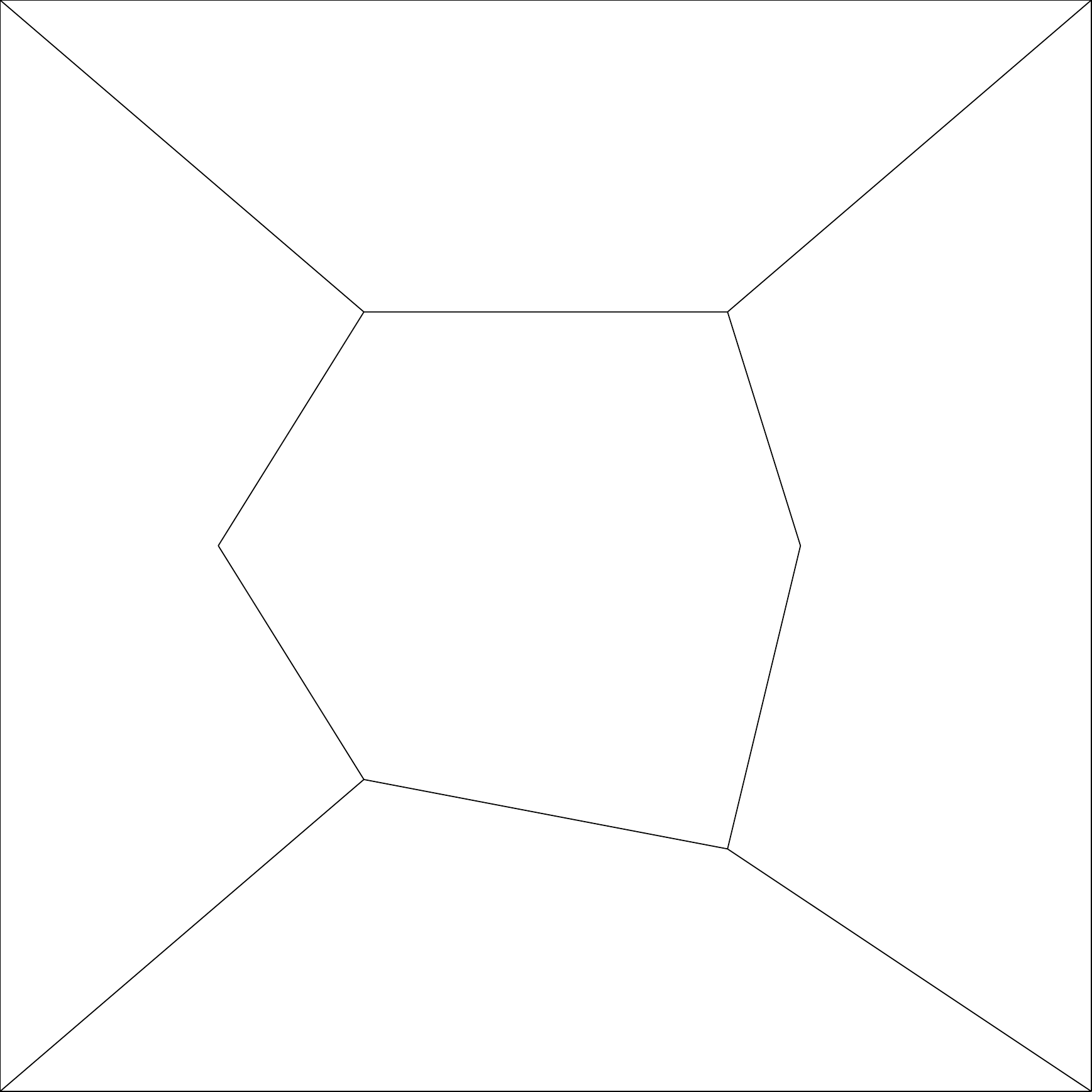} \
\includegraphics[width=1.4in]{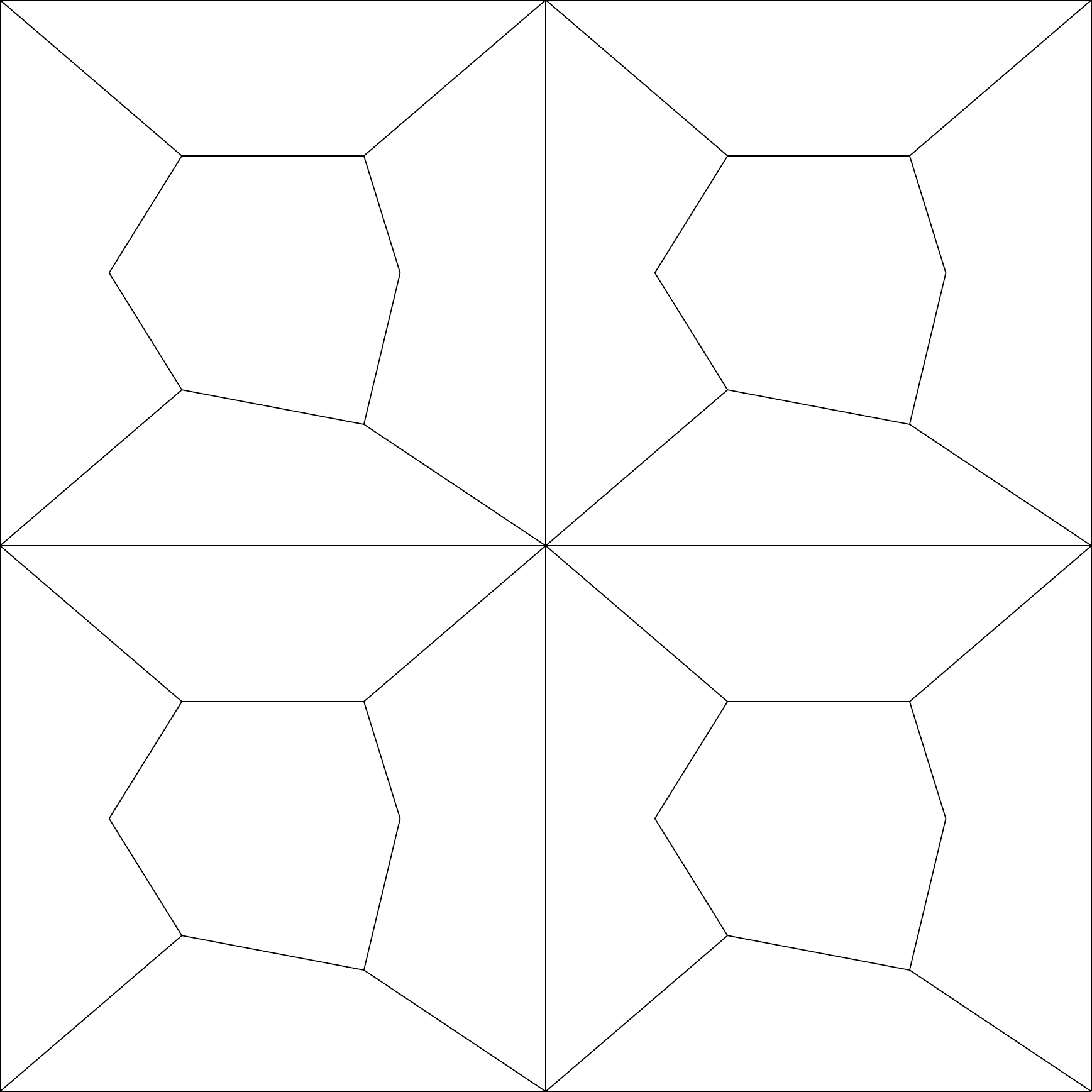} \
\includegraphics[width=1.4in]{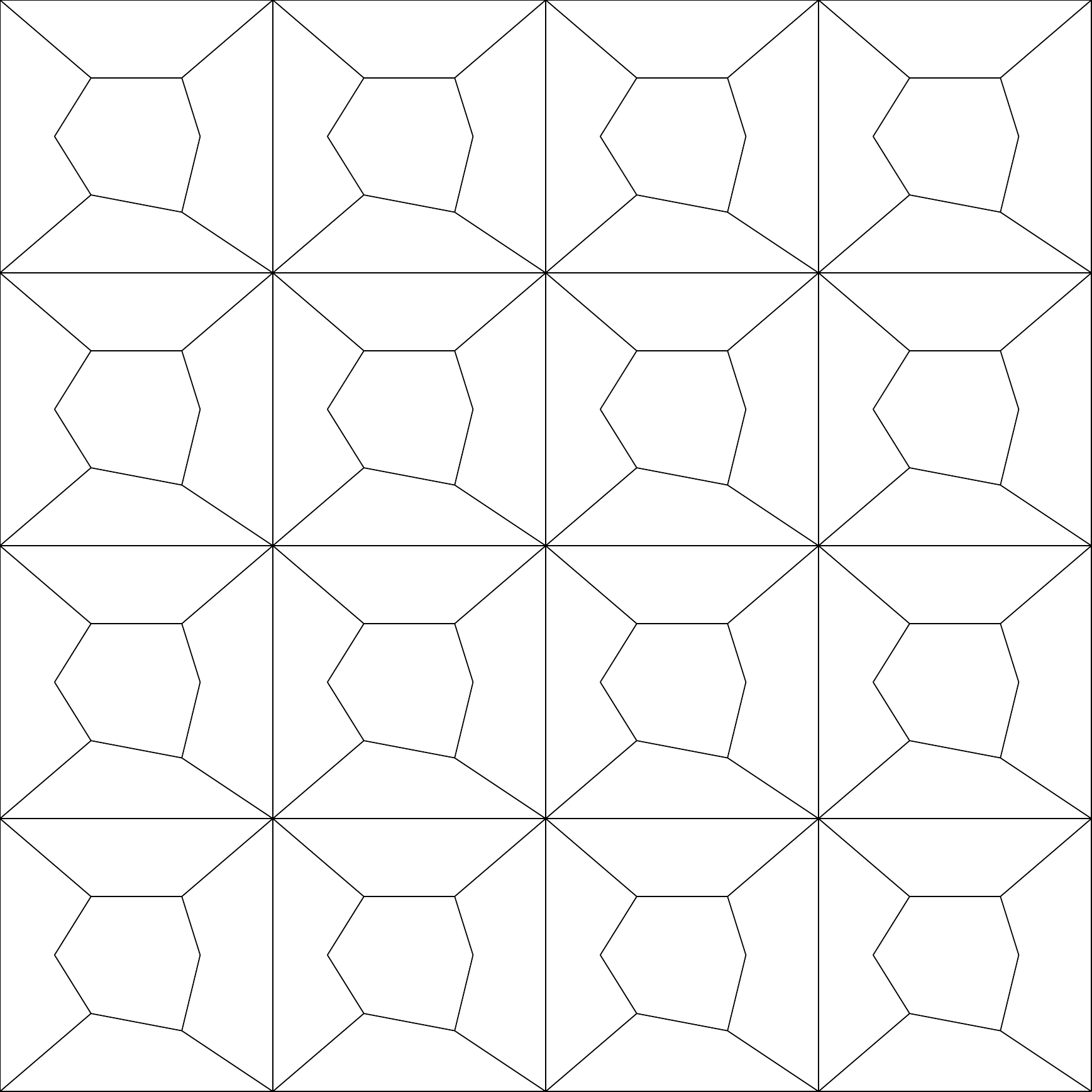}

\caption{The first three levels of quadrilateral-hexagon grids, for Table \ref{t2}.  }
\label{g-6}
\end{center}
\end{figure}

Next we solve the same problem \eqref{s-1} on a type of grids with quadrilaterals and hexagons,
 shown in Figure \ref{g-6}.
We list the result of computation in Table \ref{t2} where all theoretic convergence
  results are matched.

\begin{table}[h!]
  \centering   \renewcommand{\arraystretch}{1.05}
  \caption{ Error profiles and convergence rates on polygonal grids shown in Figure \ref{g-6} for \eqref{s-1}. }
\label{t2}
\begin{tabular}{c|cc|cc}
\hline
level & $\|Q_h u-  u_h \|_0 $  &rate &  $\3bar Q_h u- u_h \3bar $ &rate   \\
\hline
 &\multicolumn{4}{c}{by the $P_0$-$P_1$($\Lambda_0$) WG element} \\ \hline
 6&   0.2524E-03 &  2.00&   0.1313E-02 &  2.00 \\
 7&   0.6315E-04 &  2.00&   0.3282E-03 &  2.00 \\
 8&   0.1579E-04 &  2.00&   0.8204E-04 &  2.00 \\
\hline
 &\multicolumn{4}{c}{by the $P_1$-$P_2$($\Lambda_1$) WG element} \\ \hline
 6&   0.4117E-07 &  4.00&   0.1620E-04 &  3.00 \\
 7&   0.2574E-08 &  4.00&   0.2025E-05 &  3.00 \\
 8&   0.1615E-09 &  3.99&   0.2531E-06 &  3.00 \\
 \hline
 &\multicolumn{4}{c}{by the $P_2$-$P_3$($\Lambda_2$) WG element} \\ \hline
 4&   0.3371E-06 &  4.98&   0.7750E-04 &  3.99 \\
 5&   0.1058E-07 &  4.99&   0.4849E-05 &  4.00 \\
 6&   0.3417E-09 &  4.95&   0.3032E-06 &  4.00 \\
\hline
 &\multicolumn{4}{c}{by the $P_3$-$P_4$($\Lambda_3$) WG element} \\ \hline
 2&   0.5538E-04 &  4.17&   0.4462E-02 &  3.32 \\
 3&   0.8817E-06 &  5.97&   0.1414E-03 &  4.98 \\
 4&   0.1382E-07 &  6.00&   0.4436E-05 &  4.99 \\
 \hline
\end{tabular}%
\end{table}%

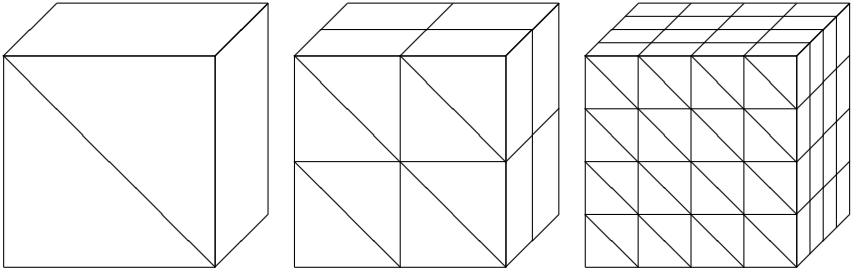
\begin{figure}[h!]
\begin{center}
 \setlength\unitlength{1pt}
    \begin{picture}(320,118)(0,3)
    \put(0,0){\begin{picture}(110,110)(0,0)
       \multiput(0,0)(80,0){2}{\line(0,1){80}}  \multiput(0,0)(0,80){2}{\line(1,0){80}}
       \multiput(0,80)(80,0){2}{\line(1,1){20}} \multiput(0,80)(20,20){2}{\line(1,0){80}}
       \multiput(80,0)(0,80){2}{\line(1,1){20}}  \multiput(80,0)(20,20){2}{\line(0,1){80}}
    \put(80,0){\line(-1,1){80}}% \put(80,0){\line(1,5){20}}\put(80,80){\line(-3,1){60}}
      \end{picture}}
    \put(110,0){\begin{picture}(110,110)(0,0)
       \multiput(0,0)(40,0){3}{\line(0,1){80}}  \multiput(0,0)(0,40){3}{\line(1,0){80}}
       \multiput(0,80)(40,0){3}{\line(1,1){20}} \multiput(0,80)(10,10){3}{\line(1,0){80}}
       \multiput(80,0)(0,40){3}{\line(1,1){20}}  \multiput(80,0)(10,10){3}{\line(0,1){80}}
    \put(80,0){\line(-1,1){80}}% \put(80,0){\line(1,5){20}}\put(80,80){\line(-3,1){60}}
       \multiput(40,0)(40,40){2}{\line(-1,1){40}}
       %  \multiput(80,40)(10,-30){2}{\line(1,5){10}}
       %  \multiput(40,80)(50,10){2}{\line(-3,1){30}}
      \end{picture}}
    \put(220,0){\begin{picture}(110,110)(0,0)
       \multiput(0,0)(20,0){5}{\line(0,1){80}}  \multiput(0,0)(0,20){5}{\line(1,0){80}}
       \multiput(0,80)(20,0){5}{\line(1,1){20}} \multiput(0,80)(5,5){5}{\line(1,0){80}}
       \multiput(80,0)(0,20){5}{\line(1,1){20}}  \multiput(80,0)(5,5){5}{\line(0,1){80}}
    \put(80,0){\line(-1,1){80}}% \put(80,0){\line(1,5){20}}\put(80,80){\line(-3,1){60}}
       \multiput(40,0)(40,40){2}{\line(-1,1){40}}
       %  \multiput(80,40)(10,-30){2}{\line(1,5){10}}
       %  \multiput(40,80)(50,10){2}{\line(-3,1){30}}

       \multiput(20,0)(60,60){2}{\line(-1,1){20}}   \multiput(60,0)(20,20){2}{\line(-1,1){60}}
       %  \multiput(80,60)(15,-45){2}{\line(1,5){5}} \multiput(80,20)(5,-15){2}{\line(1,5){15}}
       %  \multiput(20,80)(75,15){2}{\line(-3,1){15}}\multiput(60,80)(25,5){2}{\line(-3,1){45}}
      \end{picture}}

    \end{picture}
    \end{center}
\caption{  The first three levels of wedge grids used in Table \ref{t3}. }
\label{grid3}
\end{figure}

We solve a 3D problem \eqref{pde}--\eqref{bc} on the unit cube domain
  $\Omega=(0,1)^3$ with the exact
  solution
\an{ \label{s-2}  u&=\sin(\pi x) \sin(\pi y)\sin(\pi z). }

 Here we use a uniform wedge-type (polyhedron with 2 triangular faces and 3 rectangular faces)
   grids,  shown in Figure \ref{grid3}.   Here each wedge is subdivided in to three
   tetrahedrons with three rectangular faces being cut each in to two triangles,
    when defining a piecewise polynomial space $\Lambda_k$ for the
  weak gradient.
 The results are listed in Table \ref{t3},  confirming the two-order superconvergence in
  the two norms for all polynomial-degree $k\ge 1$ elements.

\begin{table}[h!]
  \centering   \renewcommand{\arraystretch}{1.05}
  \caption{ Error profiles and convergence rates on grids shown in Figure \ref{grid3} for \eqref{s-2}. }
\label{t3}
\begin{tabular}{c|cc|cc}
\hline

level & $\|Q_h u-  u_h \|_0 $  &rate &  $\3bar Q_h u- u_h \3bar $ &rate   \\
\hline
 &\multicolumn{4}{c}{by the $P_0$-$P_1$($\Lambda_0$) WG element} \\ \hline
 4&     0.0101282&1.8&     0.1286817&2.0 \\
 5&     0.0026250&1.9&     0.0324419&2.0 \\
 6&     0.0006623&2.0&     0.0081278&2.0 \\
\hline
 &\multicolumn{4}{c}{by the $P_1$-$P_2$($\Lambda_1$) WG element} \\ \hline
 4&     0.0001608&3.9&     0.0250022&3.0 \\
 5&     0.0000102&4.0&     0.0031359&3.0 \\
 6&     0.0000006&4.0&     0.0003923&3.0 \\
 \hline
 &\multicolumn{4}{c}{by the $P_2$-$P_3$($\Lambda_0$) WG element} \\ \hline
 3&     0.0003973&4.9&     0.0701007&3.9 \\
 4&     0.0000126&5.0&     0.0044450&4.0 \\
 5&     0.0000004&5.0&     0.0002788&4.0 \\
\hline
 &\multicolumn{4}{c}{by the $P_3$-$P_4$($\Lambda_0$) WG element} \\ \hline
 3&    0.7113E-04&5.9&    0.1988E-01&4.9 \\
 4&    0.1126E-05&6.0&    0.6295E-03&5.0 \\
 5&    0.1767E-07&6.0&    0.1974E-04&5.0 \\
 \hline
\end{tabular}%
\end{table}%

\end{document}